\documentclass[11pt]{article}

\usepackage[utf8]{inputenc}
\usepackage{amsmath, amsthm, amssymb, bm, txfonts, yhmath}
\usepackage[T1]{fontenc}        
\usepackage{lmodern}         
\usepackage[english]{babel}
\usepackage{indentfirst}
\usepackage{graphicx}
\usepackage{titlesec}
\usepackage{caption}
\usepackage{subcaption}
\usepackage{listings}
\usepackage{xcolor}
\usepackage{comment}
\usepackage[nottoc]{tocbibind}
\usepackage{hyperref}
\usepackage{pdfpages}
\usepackage{url}
\usepackage{authblk}
\usepackage{arydshln}
\usepackage{csquotes}
\newcommand{\RR}{\mathbb{R}}

\newcommand{\NN}{\mathbb{N}}

\newtheorem{theorem}{Theorem}
\newtheorem{lemma}{Lemma}
\newtheorem{prop}{Proposition}
\newtheorem{defi}{Definition}

\newcounter{example}

\newtheorem{remark}{Remark}
\newtheorem{assum}{Assumption}

\title{About identifiability and observability for a class of dynamical systems}
\author[1]{Alicja B. Kubik\thanks{akubik@ucm.es}}
\author[2]{Alain Rapaport\thanks{alain.rapaport@inrae.fr}}
\author[1]{Benjamin Ivorra\thanks{ivorra@ucm.es}}
\author[1]{Ángel M. Ramos\thanks{angel@mat.ucm.es}}
\affil[1]{MOMAT, Instituto de Matem\'atica Interdisciplinar (IMI), Univ. Complutense de Madrid, 28040 Madrid, Spain}
\affil[2]{MISTEA, Univ.
Montpellier, INRAE, Institut Agro, 34060 Montpellier, France}

\date{}

\begin{document}

\maketitle

\begin{center}
\textit{This document is a preprint version.}
\end{center}

\begin{abstract}                          
In this note, we propose a novel approach for a class of autonomous dynamical systems that allows, given some observations of the solutions, to identify its parameters and reconstruct the state vector. This approach relies on proving the linear independence between some functions depending on the observations and its derivatives. In particular, we show that, in some cases, only low-order derivatives are necessary, opposed to classical approaches that need more derivation. We also provide different constructive procedures to retrieve the unknowns, which are based on the resolution of some linear systems. Moreover, under some analyticity conditions, these unknowns may be retrieved with very few data. We finally apply this approach to some illustrative examples.
\end{abstract}


\section{Introduction}

In the following lines, we consider autonomous systems of ODEs, together with some observations, and will study their identifiability and observability.

When a system is identifiable or observable, we can make use of several techniques (see, for example, \cite{Gauthier1992}, \cite{Goodwin1977}, \cite{Kalman1960}, \cite{Luenberger1966}, \cite{Pronzato1997}) that may help us estimate practically the unknowns with a certain accuracy, but do not guarantee obtaining the exact values. Sometimes, one can treat the system algebraically and try to reconstruct exactly the unknowns in terms of, for example, the derivatives of the data whenever they exist and are known (or can be computed) perfectly (see \cite[Chapter 3]{Rapaport2024}). The basis of this paper is settled on \cite{Rapaport2024}, providing new methodologies not considered before, up to our knowledge.

Consider $f$ is the function that describes the dynamics of the ODE system we are considering and $h(x,\theta)$ is the function that describes the observations in terms of the solution $x\in\RR^n$ and the parameter vector $\theta\in\RR^m$. The usual approach consists on studying the injectivity of the function $\Gamma_r: (x,\theta)\mapsto  (h(x,\theta), L_f h(x,\theta), ..., L_f^{r} h(x,\theta))$, for some $r\in\NN$; in general, several works on differential algebra have been developed for this purpose (see, for instance, \cite{Rapaport2023}, \cite{Jain2019}, \cite{Ovchinnikov2021}, \cite{Saccomani2003}). In particular, when the dynamics are nonlinear, $r$ may be greater than $n+m-1$, and it may be very difficult studying this function. Through this amount of differentiation, some authors focus on expressions which are linear on all the parameters (see \cite{Gevers2016}, \cite{Ljung1990}, \cite{Ljung1994}).

We present some results for a class of systems of ODEs for which we can obtain some linear relations between functions of the parameters and functions of the observed data and its derivatives, usually avoiding the amount of Lie derivatives required in the processes previously described. If the required hypotheses are satisfied, we prove that these systems are identifiable and/or observable. Moreover, we provide constructive ways to recover the unknowns, which are based on solving some linear systems of equations. To do this, the proofs will mainly require linear independence of some sets of functions. This presented approach will be illustrated through some examples in Section \ref{illexmp}.

\section{A general framework} \label{gen_model}

Consider the system
\begin{equation}\label{general_sys}
    \begin{array}{lcl}
        \dot{x}(t;\xi,\theta) & = & f(x(t;\xi,\theta),\theta), \quad x(0;\xi,\theta) = \xi, \\[0.4em]
        y_{(\xi,\theta)}(t) & = & h(x(t;\xi,\theta),\theta),
    \end{array}
\end{equation}
where $f(\cdot,\cdot) : \Omega\times\Theta \rightarrow \RR^n$ is a function of $(x,\theta)$ which is locally Lipschitz-continuous w.r.t. $x \in \Omega \subset \RR^n$ and continuous w.r.t. $\theta \in \Theta \subset \RR^b$; $\theta$ are the constant parameters of the system; $\Omega$ is a positively invariant set with respect to the system of ODEs of System \eqref{general_sys}; $x(\cdot; \xi,\theta): \mathcal{I} \rightarrow \Omega$ denotes the unique solution of the system of ODEs of System \eqref{general_sys} with initial condition $\xi \in \Omega$ and we assume it is globally defined, i.e., $\mathcal{I} = [0,+\infty)$; and the output $y_{(\xi,\theta)}(t)$, $t \in \mathcal{S} \subset \mathcal{I}$, is described by some known function $h(\cdot,\cdot) : \Omega\times\Theta \rightarrow \RR^m$.

We aim to know if, given the output $y_{(x_0,\theta_0)}$, for some $(x_0,\theta_0)$, we can determine univocally this pair $(x_0,\theta_0)$ that produces this output.
 
\begin{defi}[Identifiability in a set]\label{identifiability_set}
    System \eqref{general_sys} is identifiable on $\Theta$ in $\mathcal{S} \subset \mathcal{I}$  with initial conditions in $\Omega$ whether, for any $\xi \in \Omega$, given different $\theta_1, \theta_2 \in \Theta$, there exists some time $t \in  \mathcal{S}$ such that
    $$h(x(t; \xi,\theta_1),\theta_1) \neq h(x(t; \xi,\theta_2),\theta_2).$$
    Equivalently, if $h(x(t; \xi,\theta_1),\theta_1) = h(x(t; \xi,\theta_2),\theta_2)$, for all $t \in \mathcal{S}$ and any $\xi \in \Omega$, implies that $\theta_1 = \theta_2$.
\end{defi}

\begin{defi}[Observability in a set]\label{observability_set}
    System \eqref{general_sys} is observable on $\Omega$ in $\mathcal{S} \subset \mathcal{I}$ with parameters in $\Theta$ whether, for any $\theta \in \Theta$, given different $x_0, x_0' \in \Omega$, there exists some time $t \in \mathcal{S}$ such that
    $$h(x(t; x_0,\theta),\theta) \neq h(x(t; x_0',\theta),\theta).$$
    Equivalently, if $h(x(t; x_0,\theta),\theta) = h(x(t; x_0',\theta),\theta)$, for all $t \in \mathcal{S}$ and any $\theta \in \Theta$, implies that $x_0 = x_0'$.
\end{defi}

System \eqref{general_sys} is identifiable (resp. observable) if Definition \ref{identifiability_set} (resp. Definition \ref{observability_set}) is fulfilled for $\mathcal{S} = \mathcal{I}$.

It is straightforward noticing that, given a set $\mathcal{A} \subset \mathcal{I}$, if a system is identifiable (resp. observable) in a subset $\mathcal{A}_1 \subset \mathcal{A}$, then it is identifiable (resp. observable) in $\mathcal{A}$.

Let us now, given $\theta \in \Theta$, denote $h_{\theta}(\cdot) = h(\cdot,\theta)$ and $f_{\theta}(\cdot) = f(\cdot,\theta)$, and assume $h_{\theta} \in \mathcal{C}^{d}(\Omega; \RR^m)$, $f_{\theta} \in \mathcal{C}^{\max\{0,d-1\}}(\Omega; \RR^n)$, for some $d \in \NN\cup\{0\}$. Then, $y_{(\xi,\theta)} \in \mathcal{C}^{d}(\mathcal{I}; \RR^m)$. In particular,
$$\dot{y}_{(\xi,\theta)}(t) = L^1_{f_{\theta}} h_{\theta}(x(t; \xi,\theta)), \quad \forall\, t\geq 0,$$
where we denote by $L^1_{f_{\theta}}h_{\theta}$ the Lie derivative of $h_{\theta}$ with respect to the vector field $f_{\theta}$. If we continue differentiating $y_{(\xi,\theta)}$, we have
$$y^{(k)}_{(\xi,\theta)}(t) = \dfrac{\mathrm{d}^k}{\mathrm{d} t^k} h(x(t; \xi,\theta)) = L_{f_{\theta}}^{k} h_{\theta}(x(t; \xi,\theta)), \quad \forall\, t \geq 0,\ k \in \{1,\dots,d\},$$
denoting 
$$y^{(k)} = \dfrac{\mathrm{d}^k y}{\mathrm{d} t^k}, \quad k \in \NN, \text{ when } y \in \mathcal{C}^k(\mathcal{I}).$$
Then, denoting $y^{(0)} = y$, we define the map $\mathcal{L}_{f_{\theta},h_{\theta},d} : \Omega \rightarrow \RR^{m\times d}$ as
$$\mathcal{L}_{f_{\theta},h_{\theta},d}(\xi) = \left(y^{(0)}_{(\xi,\theta)}(0), \dots, y^{(d)}_{(\xi,\theta)}(0)\right).$$

In the following, we may denote $y_{(\xi,\theta)}$, $\dot{y}_{(\xi,\theta)}$ and $y^{(k)}_{(\xi,\theta)}$ as $y$, $\dot{y}$ and $y^{(k)}$, respectively, in order to simplify the notation.

Before presenting our main results, let us give some results about linear independence of families of functions, a concept on which our approach is based.

\subsection{About linear independence}

\begin{defi}\label{defi_lindep}
    Given $\mathcal{I} \subset \RR$, functions $\phi_i : \mathcal{I} \rightarrow \RR$, $i \in \{1,\dots,q\}$, $q \in \NN$ are said to be linearly independent if the only constants $a_1, \dots, a_q\in\RR$ such that $a_1\phi_1(t) + \dots + a_q\phi_q(t) = 0$, $\forall\, t \in \mathcal{I},$ are $a_1 = \dots = a_q = 0.$
\end{defi}

\begin{lemma}\label{lindep}
    Let  $\phi_i : \mathcal{S}\subset\RR \rightarrow \RR$, $i \in \{1,\dots,q\}$, $q \in \NN$, $\mathcal{S} \subset \mathcal{I}$. These functions are linearly independent if, and only if, there exist $q$ different times $t_1, \dots, t_{q} \in \mathcal{S}$ such that the matrix $(\phi_i(t_j))_{i,j = 1,\dots,q}$ has full rank.
\end{lemma}

\begin{proof}
    We are going to prove the first implication by induction on the number of linearly independent functions, taking into account Definition \ref{defi_lindep}.

    \noindent\underline{Case $q = 2$:} If $\phi_1$ and $\phi_2$ are linearly independent in $\mathcal{S}$, then, given $t_1 \in \mathcal{S}$ such that $\phi_i(t_1) \neq 0$, for some $i \in \{1,2\}$,
    $$\det \left(\begin{array}{cc}
        \phi_1(t) & \phi_2(t) \\
        \phi_1(t_1) & \phi_2(t_1)
    \end{array}\right) =
    \phi_2(t_1)\phi_1(t) - \phi_1(t_1)\phi_2(t)$$
    is not identically null in $\mathcal{S}$; otherwise, $a_1 = \phi_2(t_1)$ and $a_2 = -\phi_1(t_1)$ would be coefficients, not both null, such that
    $$a_1\phi_1(t) + a_2\phi_2(t) = 0, \quad \forall\, t \in \mathcal{S},$$
    which is in contradiction with $\phi_1$ and $\phi_2$ being linearly independent in $\mathcal{S}$. Hence, there exists some time $t_2 \in \mathcal{S}$ such that
    $$\det \left(\begin{array}{cc}
        \phi_1(t_2) & \phi_2(t_2) \\
        \phi_1(t_1) & \phi_2(t_1)
    \end{array}\right) \neq 0.$$

    \noindent\underline{Induction step:} Let $q \geq 3$. Assume that, given $\phi_1, \dots, \phi_{q-1}$ linearly independent, there exist $t_1, \dots,$ $t_{q-1} \in \mathcal{S}$ such that
    $$D_{q-1} = \det \left((\phi_i(t_j))_{
        \begin{subarray}{l}
            i = 1,\dots,q-1 \\
            j = 1,\dots,q-1
        \end{subarray}
    }\right) \neq 0.$$
    
    If $\phi_{q}$ is a function such that $\phi_1, \dots, \phi_{q}$ are linearly independent, then
    $$\det \left(\begin{array}{cccc}
        \phi_1(t) & \cdots & \phi_{q-1}(t) & \phi_{q}(t) \\
        \phi_1(t_1) & \cdots & \phi_{q-1}(t_1) & \phi_{q}(t_1) \\
        \vdots & \ddots & \vdots & \vdots \\
        \phi_1(t_{q-1}) & \cdots & \phi_{q-1}(t_{q-1}) & \phi_{q}(t_{q-1})
    \end{array}\right) =
    (-1)^{q-1}D_{q-1}\phi_{q}(t) + d_{q-1}\phi_{q-1}(t) + \dots + d_1\phi_1(t)$$
    is not identically null in $\mathcal{S}$, where
    $$d_k = (-1)^{k-1}\det \left((\phi_i(t_j))_{
        \begin{subarray}{l}
            i = 1,\dots,k-1,k+1,\dots,q \\
            j = 1,\dots,q-1
        \end{subarray}
    }\right), \quad k \in \{1,\dots,q-1\}.$$
    Otherwise, since $D_{q-1} \neq 0$, there would exist $q$ coefficients $a_i = d_i, \, i \in \{1,\dots,q-1\}$, $a_q = (-1)^{q-1}D_{q-1}$, not all of them null, such that
    $$\sum_{i=1}^{q} a_i\phi_i(t) = 0, \quad \forall\, t \in \mathcal{S},$$
    which is in contradiction with the fact that $\phi_1, \dots, \phi_{q}$ are linearly independent. Then, there exists some time $t_q \in \mathcal{S}$ such that
    \begin{equation}\label{det_phi_q}
        \det \left((\phi_i(t_j))_{
            \begin{subarray}{l}
                i = 1,\dots,q \\
                j = 1,\dots,q
            \end{subarray}
        }\right) \neq 0,
    \end{equation}
    as we wanted to prove.

   Finally, for the second implication, assume there exist $t_1, \dots, t_q \in \mathcal{S}$ such that \eqref{det_phi_q} is satisfied. If there exist $a_1, \dots, a_q$ such that 
   $$\sum_{i=1}^q a_i\phi_i(t) = 0, \quad \forall\, t \in \mathcal{S},$$
   this would in particular imply that
   $$\sum_{i=1}^q a_i \left(\begin{array}{c}
        \phi_i(t_1)  \\
        \vdots \\
        \phi_i(t_q)
    \end{array}\right) = 0.$$
    But these vectors are linearly independent, given that their determinant is non-null, and hence $a_i = 0$, $\forall\, i \in \{1,\dots,q\}$. Thus, $\phi_1, \dots, \phi_q$ are linearly independent.
\end{proof}

\begin{remark}
    Notice that Lemma \ref{lindep} gives a necessary and sufficient condition for linear independence which requires neither analyticity nor differentiability of $\phi_1, \dots, \phi_q$, as opposed to the classical condition (see \cite{Bocher1900}) that these functions have non-null Wro\'nskian, i.e., that there exists some time $\tilde{t} \in \mathcal{S}$ such that
    \begin{equation*}
        \left.W(t)\right|_{t=\tilde{t}} = \left.\left(\det \left(\dfrac{\mathrm{d}^k \phi_i}{\mathrm{d} t^k}(t)\right)_{
            \begin{subarray}{l}
                i = 1,\dots,q \\
                k = 0,\dots,q-1
            \end{subarray}
        }\right)\right|_{t=\tilde{t}} \neq 0.
    \end{equation*}
\end{remark}

\subsection{Main results}

We recall the classical result on observability based on Lie derivatives (see \cite{Rapaport2024}, \cite{Gauthier1992}, \cite{ORC}), and hence will not provide a proof for Theorem \ref{generalization_obs}.

\begin{theorem}\label{generalization_obs}
    Let $h_{\theta,i} \in \mathcal{C}^{d_i}(\Omega; \RR^m)$, for some $d_i \in \NN\cup\{0\}$, $i \in \{1,\dots,m\}$, $h_{\theta} = (h_{\theta,1}, \dots, h_{\theta,m})$, and $f_{\theta} \in \mathcal{C}^{d-1}(\Omega; \RR^n)$, $d = \max\{1, d_1, \dots, d_m\}$, for any $\theta \in \Theta$. If
    $$\mathcal{L}_{f_{\theta},h_{\theta},\{d_1,\dots,d_m\}} \ : \ \xi \mapsto \left(\mathcal{L}_{f_{\theta},h_{\theta,1},d_1}(\xi), \dots, \mathcal{L}_{f_{\theta},h_{\theta,m},d_m}(\xi)\right)$$
    is injective in $\Omega$, then System \eqref{general_sys} is observable on $\Omega$ in any semi-open interval $[a,b) \subset \mathcal{I}$ with parameters in $\Theta$.
\end{theorem}

\begin{remark}
    If one extends the dynamics with $\dot{\theta} = 0$, then both the identifiability and observability properties can be studied as a particular case of observability in higher dimension.
\end{remark}

We present now our main result, which is an alternative to usual approaches (as the one recalled above) to check identifiability.

\begin{theorem}\label{generalization_ident}
    Let $h_{\theta,i} \in \mathcal{C}^{d_i'}(\Omega; \RR^n)$, for some $d_i' \in \NN\cup\{0\}$, $i \in \{1,\dots,m\}$, $h_{\theta} = (h_{\theta,1}, \dots, h_{\theta,m})$, and $f_{\theta} \in \mathcal{C}^{d'-1}(\Omega; \RR^n)$, $d' = \max\{1, d_1', \dots, d_m'\}$, for any $\theta \in \Theta$. Consider $\mathcal{D} \subset \RR^{d_1'+\dots+d_m'+m}$ such that the output of System \eqref{general_sys} satisfies, for any $(\xi,\theta) \in \Omega\times\Theta$, that
    $$\left(y_1^{(0)}(t), \dots, y^{(d_1')}_1(t), \dots, y_m^{(0)}(t), \dots, y^{(d_m')}_m(t)\right) \in \mathcal{D}, \quad \forall\, t \in \mathcal{I}.$$
    If there exist maps $g : \mathcal{D} \rightarrow \RR^{q+p}$ and $r : \Theta \rightarrow \RR^q$, for some $q,p \in \NN$, and a subset $\mathcal{S} \subset \mathcal{I}$ such that every connected part of $S$ contains some open interval, satisfying:
    \begin{enumerate}
        \item $g = (g_{1,0}, \dots, g_{1,q_1}, \dots, g_{p,0}, \dots, g_{p,q_p})$ and $r = (r_{1,1}, \dots, r_{1,q_1}, \dots, r_{p,1}, \dots, r_{p,q_p})$, with $q_1 + \dots + q_p = q$, satisfy that
        \begin{equation}\label{linear_eq}
        g_{j,0}(y_1^{(0)}(t), \dots, y^{(d_m')}_m(t)) = \sum_{l=1}^{q_j} r_{j,l}(\theta) g_{j,l}(y_1^{(0)}(t), \dots, y^{(d_m')}_m(t)),
        \end{equation}
        for all $t \in \mathcal{S}$, $j \in \{1,\dots,p\}$, for any $(\xi,\theta) \in \Omega\times\Theta$,
        \item $r$ is injective, and
        \item for any $j \in \{1,\dots,p\}$ and $(\xi,\theta) \in \Omega\times\Theta$, we have that $g_{j,l}(y_1^{(0)}(t), \dots, y^{(d_m')}_m(t))$, for $l \in \{1,\dots,q_j\}$, are linearly independent functions with respect to $t \in \mathcal{S}$,
    \end{enumerate}then System \eqref{general_sys} is identifiable on $\Theta$ in $\mathcal{S}$ with initial conditions in $\Omega$.
\end{theorem}

\begin{proof} Given $\xi \in \Omega$, let $\theta_1, \theta_2 \in \Theta$ such that
$$h(x(t; \xi,\theta_1),\theta_1) = h(x(t; \xi,\theta_2),\theta_2), \quad \forall\, t \in \mathcal{S},$$
i.e.,
$$y_{(\xi,\theta_1)}(t) = y_{(\xi,\theta_2)}(t), \quad \forall\, t \in \mathcal{S}.$$
Then, since every connected part of $\mathcal{S}$ contains some open interval, this implies that
$$y^{(k)}_{(\xi,\theta_1),i}(t) = y^{(k)}_{(\xi,\theta_2),i}(t), \quad \forall\, t \in \mathcal{S}, \ k \in \{0,\dots,d_i'\}, \ i \in \{1,\dots,m\}.$$
Now, since
$$g_{j,0}(y_{(\xi,\theta_1),1}^{(0)}, \dots, y_{(\xi,\theta_1),m}^{(d_m')}) - g_{j,0}(y_{(\xi,\theta_2),1}^{(0)}, \dots, y_{(\xi,\theta_2),m}^{(d_m')}) \equiv 0$$
in $\mathcal{S}$, from \eqref{linear_eq}, we obtain $$\sum_{l=1}^{q_j} \left(r_{j,l}(\theta_1) - r_{j,l}(\theta_2)\right) g_{j,l}(y_{(\xi,\theta_1),1}^{(0)}, \dots, y_{(\xi,\theta_1),m}^{(d_m')}) \equiv 0,$$
in $\mathcal{S}$, for every $j \in \{1,\dots,p\}$. Given the linear independence of $g_{j,l}(y_1^{(0)}, \dots, y_m^{(d_m')})$, $l \in \{1,\dots,q_j\}$, in $\mathcal{S} \subset \mathcal{I}$, for every $j \in \{1,\dots,p\}$, $(\xi,\theta) \in \Omega$, then
$$r(\theta_1) = r(\theta_2).$$
Since $r$ is an injective function, this implies that $\theta_1 = \theta_2$. Hence, System \eqref{general_sys} is identifiable on $\Theta$ in $\mathcal{S} \subset \mathcal{I}$ with initial conditions in $\Omega$.
\end{proof}

Next we will see that, assuming some assumptions are satisfied, if we know $y_{(x_0,\theta_0)}$ in some time set, then we are able to recover the unknowns $(x_0,\theta_0) \in \Omega\times\Theta$.

\begin{theorem}\label{det_univocally}
    Assume we know $y_{(x_0,\theta_0)}(t)$, $t \in \mathcal{S} \subset \mathcal{I}$, $\mathcal{S}$ such that every connected component contains an open interval. If the hypotheses of Theorems \ref{generalization_obs} and \ref{generalization_ident} are satisfied, then we can reconstruct the pair $(x_0,\theta_0)$ univocally using the values of $y_{(x_0,\theta_0)}$ and its derivatives at, at most, $q+1 = q_1 + \dots + q_p + 1$ suitable values of $t \in \mathcal{S}$.
\end{theorem}

\begin{proof}  Let $\phi_{j,l}(t) = g_{j,l}(y_{(x_0,\theta_0),1}^{(0)}(t), \dots, y^{(d_m')}_{(x_0,\theta_0),m}(t))$, $t \in \mathcal{S}$, $l \in \{0,\dots,q_j\},\ j \in \{1,\dots,p\}$. Then, by hypothesis, $\phi_{j,1}, \dots, \phi_{j,q_j}$ are linearly independent in $\mathcal{S}$, for each $j \in \{1,\dots,p\}$. Hence, as seen in Lemma \ref{lindep}, for every $j \in \{1,\dots,p\}$, there exist $q_j$ different times $t_{j,1}, \dots, t_{j,q_j} \in \mathcal{S}$ such that
\begin{equation*}
    \det \left((\phi_{j,l}(t_{j,\ell}))_{
        \begin{subarray}{l}
            l=1,\dots,q_j \\
            \ell=1,\dots,q_j
        \end{subarray}
    }\right) \neq 0.
\end{equation*}
Then, there exists a unique solution $\sigma$ to 
\begin{equation}\left(
    \begin{array}{ccc}
        \phi_{j,1}(t_{j,1}) & \cdots & \phi_{j,q_j}(t_{j,1}) \\
        \vdots & \ddots & \vdots \\
        \phi_{j,1}(t_{j,q_j}) & \cdots & \phi_{j,q_j}(t_{j,q_j})
    \end{array}\right)\left(
    \begin{array}{c}
        \sigma_{j,1} \\
         \vdots \\
         \sigma_{j,q_j}
    \end{array}\right) = \left(
    \begin{array}{c}
        \phi_{j,0}(t_{j,1}) \\
        \vdots \\
        \phi_{j,0}(t_{j,q_j})
    \end{array}\right).
\end{equation}
Since it is unique, attending to \eqref{linear_eq}, it fulfills $\sigma_{j,l} = r_{j,l}(\theta_0)$, $l \in \{1,\dots,q_j\},\ j \in \{1,\dots,p\}$. Taking into account that we consider $r$ to be an injective function, such that $r^{-1} : r(\Theta) \rightarrow \Theta$, we may hence recover our original parameter vector $\theta_0$ as
\begin{equation}\label{iii}
    \theta_0 = r^{-1}(\sigma_{1,1}, \dots, \sigma_{p,q_p}).
\end{equation}
Finally, to recover the initial condition, take some time $\tilde{t} \in \mathcal{S}$, which can be some $\tilde{t} \in \{t_{1,1}, \dots, t_{p,q_p}\}$. Due to the injectivity of $\mathcal{L}_{f_{\theta_0},h_{\theta_0},\{d_1,\dots,d_m\}}$ in $\Omega$, there exists a unique $\tilde{\xi} \in \Omega$ such that
\begin{equation}\label{xitilde}
    \tilde{\xi} = \mathcal{L}^{-1}_{f_{\theta_0},h_{\theta_0},\{d_1,\dots,d_m\}}(y_{(x_0,\theta_0),1}^{(0)}(\tilde{t}), \dots, y_{(x_0,\theta_0),m}^{(d_m)}(\tilde{t})),
\end{equation}
noticing that $y_{(x_0,\theta_0),i}^{(k)}(\tilde{t}) = y_{(\tilde{\xi},\theta_0),i}^{(k)}(0)$, for all $i \in \{1,\dots,k\}$, $k \in \{0,\dots,d_i\}$. We can recover the initial condition in $\Omega$ integrating backwards the ODE system in System \eqref{general_sys} knowing $\tilde{\xi}$, $\tilde{t}$ and $\theta_0$ (we can do it because $f_{\theta_0}$ is Lipschitz in $\Omega$ positively invariant w.r.t. the ODE system given in \eqref{general_sys}). If $0 \in \mathcal{S}$, this part may be performed straightforwardly choosing $\tilde{t} = 0$.

This is, we have recovered $\theta_0$ and $x_0$ from the data univocally knowing $y_{(x_0,\theta_0)}(t)$, for all $t \in \mathcal{S}$, using its values and the values of its derivatives at $q+1$ (at most) different times.
\end{proof}

Therefore, given a system of first order autonomous ODEs, along with some observations, we can check the hypotheses in Theorems \ref{generalization_obs} and \ref{generalization_ident} in order to determine the observability and/or identifiability of our model. Then, if the hypotheses in Theorem \ref{det_univocally} are satisfied, we can recover the initial condition and parameter vector.

\begin{remark}\label{range_qi_times}
Notice that, in order to be able to recover $(x_0,\theta_0)$ following the procedure in the proof of Theorem \ref{det_univocally}, for each set $\{\phi_{j,1}, \dots, \phi_{j,q_j}\}$, $j \in \{1,\dots,p\}$, we need to find $q_j$ different suitable times. However, some of these times may coincide among different sets of linearly independent functions. Thus, the quantity of different times we need to find is between $\tilde{q} = \max\{q_1, \dots, q_p\}$ and $q = q_1 + \dots + q_p$, along with maybe $\tilde{t} = 0$, which we can use to recover the initial condition if $0 \in \mathcal{S}$ and could be one of the other times.

Recall, moreover, that we do not necessarily need $y_{(x_0,\theta_0)}(t)$, for all $t \in \mathcal{S}$, but it would be sufficient having the values of $y_{(x_0,\theta_0)}$ and its derivatives at the aforementioned different times, where the order of the derivatives that we need are the same as for Theorem \ref{det_univocally}.
\end{remark}

The times required in the proof of Theorem \ref{det_univocally} may be anywhere in $\mathcal{S} \subset \mathcal{I}$, and hence may be difficult to find in practice. In the following Lemma \ref{infinitesimal_ident} we give sufficient hypotheses such that we can identify System \eqref{general_sys} in any semi-open subset  $[a,b) \subset \mathcal{S}$, similarly as considered in Theorem \ref{generalization_obs}; this will imply that we will be able to choose this set of times in any open of these semi-open intervals. 

\begin{lemma}\label{infinitesimal_ident}
   Let us assume that the hypotheses of Theorem \ref{generalization_ident} are satisfied for some $\mathcal{S} \subset \mathcal{I}$ and, for any $j \in \{1,\dots,p\}$, $l \in \{1,\dots,q_j\}$ and $(\xi,\theta) \in \Omega\times\Theta$, the functions $g_{j,l}(y_1^{(0)}(t), \dots, y^{(d_m')}_m(t))$ are analytic functions with respect to $t \in \mathcal{I}$. Then, System \eqref{general_sys} is identifiable on $\Theta$ in any semi-open interval $[a,b) \subset \mathcal{S}$ with initial conditions in $\Omega$. Moreover, if $\mathcal{S}$ is connected, it is enough asking that $g_{j,l}(y_1^{(0)}(t), \dots, y_m^{(d_m')}(t))$, $j \in \{1,\dots,p\}$, $l \in \{1,\dots,q_j\}$, are analytic with respect to $t \in \mathcal{S}$.
\end{lemma}

\begin{proof}
    Given $\xi \in \Omega$, let $\theta_1, \theta_2 \in \Theta$ such that
    $$h(x(t; \xi,\theta_1),\theta_1) = h(x(t; \xi,\theta_2),\theta_2), \quad \forall\, t \in [a,b),$$
    i.e.,
    $$y_{(\xi,\theta_1)}(t) = y_{(\xi,\theta_2)}(t), \quad \forall\, t \in [a,b).$$
    This implies that
    $$y^{(k)}_{(\xi,\theta_1),i}(t) = y^{(k)}_{(\xi,\theta_2),i}(t), \quad \forall\, t \in [a,b),\ k \in \{0,\dots,d_i'\},\ i \in \{1,\dots,m\}.$$
    Given that, for any $j \in \{1,\dots,p\}$, $l \in \{1,\dots,q_j\}$ and $(\xi,\theta) \in \Omega\times\Theta$, the functions $g_{j,l}(y_1^{(0)}, \dots,$ $y^{(d_m')}_m)$ are analytic in $\mathcal{I}$, then the function
    $$G_{j,\theta}(y_1^{(0)}, \dots, y^{(d_m')}_m) = \sum_{l=1}^{q_j} r_{j,l}(\theta)g_{j,l}(y_1^{(0)}, \dots, y^{(d_m')}_m)$$
    is also analytic in $\mathcal{I}$. Since $[a,b) \subset \mathcal{S}$,
    $$G_{j,\theta}(y_1^{(0)}, \dots, y_m^{(d_m')}) = g_{j,0}(y_1^{(0)}, \dots, y_m^{(d_m')}),$$
    in $[a,b)$ for any $\theta \in \Theta$, and
    $$g_{j,0}(y_{(\xi,\theta_1),1}^{(0)}, \dots, y_{(\xi,\theta_1),m}^{(d_m')}) = g_{j,0}(y_{(\xi,\theta_2),1}^{(0)}, \dots, y_{(\xi,\theta_2),m}^{(d_m')})$$
    in $[a,b)$, then
    $$G_{j,\theta_1}(y_{(\xi,\theta_1),1}^{(0)}, \dots, y_{(\xi,\theta_1),m}^{(d_m')}) = G_{j,\theta_2}(y_{(\xi,\theta_2),1}^{(0)}, \dots, y_{(\xi,\theta_2),m}^{(d_m')})$$
    in $[a,b)$. This implies that, for every $j \in \{1,\dots,p\}$,
    $$R_j = \sum_{l=1}^{q_j} (r_{j,l}(\theta_1) - r_{j,l}(\theta_2))g_{j,l}(y_{(\xi,\theta_1),1}^{(0)}, \dots, y_{(\xi,\theta_1),m}^{(d_m')}) \equiv 0$$
    in $[a,b)$. Due to the analyticity of $g_{j,l}(y_1^{(0)}, \dots, y_m^{(d_m')})$, $l \in \{1,\dots,q_j\}$, in $\mathcal{I}$, then $R_j$ is also analytic in $\mathcal{I}$. Then, if $R_j \equiv 0$ in $[a,b)$, we have that $R_j \equiv 0$ in $\mathcal{I}$ (\cite[Theorem 8.5]{Rudin1976}). Therefore, since $g_{j,l}(y_1^{(0)}, \dots, y_m^{(d_m')})$, $l \in \{1,\dots,q_j\}$, are linearly independent in $\mathcal{S} \subset \mathcal{I}$, they are in particular linearly independent in $\mathcal{I}$, and hence, in order for $R_j$ to be 0 in $\mathcal{I}$, for all $j \in \{1,\dots,p\}$, we need
    $$r(\theta_1) = r(\theta_2).$$
    Since $r$ is an injective function, this implies that $\theta_1 = \theta_2$.

    Notice that, if $\mathcal{S}$ is connected, $g_{j,l}(y_1^{(0)}, \dots, y_m^{(d_m')})$, $l \in \{1,\dots,q_j\}$, analytic in $\mathcal{S}$ implies $R_j$ is also analytic in $\mathcal{S}$, and $R_j \equiv 0$ in $[a,b) \subset \mathcal{S}$ connected implies $R_j \equiv 0$ in $\mathcal{S}$ (if $\mathcal{S}$ is not connected, we can only assure $R_j \equiv 0$ in the connected component of $\mathcal{S}$ containing $[a,b)$). Therefore, we conclude analogously using the linear independence of $g_{j,l}(y_1^{(0)}, \dots, y_m^{(d_m')})$, $l \in \{1,\dots,q_j\}$, in $\mathcal{S}$.
    
    Hence, System \eqref{general_sys} is identifiable on $\Theta$ in any $[a,b) \subset \mathcal{S}$ with initial conditions in $\Omega$.
\end{proof}

Taking into account Theorem \ref{generalization_obs} and Lemma \ref{infinitesimal_ident}, we can try to recover $x_0$ and $\theta_0$ knowing $y_{(x_0,\theta_0)}$ only in some $[a,b) \subset \mathcal{I}$. This will be shown in the following Lemma \ref{recover_infi_lindep}.

\begin{lemma}\label{recover_infi_lindep}
    Assume we know $y_{(x_0,\theta_0)}$ in some $[a,b) \subset \mathcal{I}$. If the hypotheses of Theorem \ref{generalization_obs} and Lemma \ref{infinitesimal_ident} are satisfied for some $\mathcal{S} \subset \mathcal{I}$ such that $[a,b) \subset \mathcal{S}$, then we can reconstruct $(x_0,\theta_0)$ univocally using the values of $y_{(x_0,\theta_0)}$ and its derivatives at a finite amount of suitable values of $t \in [a,b)$. Actually, the needed number of values with this procedure is between $\tilde{q} = \max\{q_1, \dots, q_p\}$ and $q = q_1 + \dots + q_p$.
\end{lemma}

\begin{proof}
    The way to recover the parameters $\theta_0$ is analogous to the proof of Theorem \ref{det_univocally}. Indeed, we only need to see that, given $j \in \{1,\dots,p\}$, since the functions $g_{j,l}(y_1^{(0)}, \dots, y^{(d_m')}_m)$, $l \in \{1,\dots,q_j\}$, are linearly independent in some $\mathcal{S} \subset \mathcal{I}$ and analytic in $\mathcal{I}$, they are linearly independent also in $[a,b)$.
    
    Given $j \in \{1,\dots,p\}$, assume that $g_{j,l}(y_1^{(0)}, \dots,$ $y^{(d_m')}_m)$, $l \in \{1,\dots,q_j\}$, are linearly dependent in $[a,b)$, i.e., there exist $a_{j,1}, \dots, a_{j,q_j} \in \RR$ not all of them null such that
    $$G_j(t) = \sum_{l=1}^{q_j} a_{j,l}g_{j,l}(y_1^{(0)}(t), \dots, y^{(d_m')}_m(t)) = 0, \ \forall\, t \in [a,b).$$
    Since $g_{j,l}(y_1^{(0)}, \dots, y^{(d_m')}_m)$, $l \in \{1,\dots,q_j\}$, are analytic in $\mathcal{I}$, then so it is $G_j$. Hence, because of \cite[Theorem 8.5]{Rudin1976}, this implies that $G_j \equiv 0$ in $\mathcal{I}$ and, thus, $g_{j,l}(y_1^{(0)}, \dots, y^{(d_m')}_m)$, $l \in \{1,\dots,q_j\}$, are linearly dependent in $\mathcal{I}$, and hence in $\mathcal{S}$, which is a contradiction. Moreover, if $\mathcal{S}$ is connected, it is enough asking for $g_{j,l}(y_1^{(0)}, \dots, y^{(d_m')}_m)$, $l \in \{1,\dots,q_j\}$, analytic in $\mathcal{S}$, since, hence, so it is $G_j$ and, again because of \cite[Theorem 8.5]{Rudin1976}, $G_j(t) = 0$ for $t \in [a,b) \subset \mathcal{S}$ connected implies $G_j \equiv 0$ in $\mathcal{S}$, which leads to the same contradiction.

    Therefore, for each $j \in \{1,\dots,p\}$, the functions $\phi_{j,1}, \dots,$ $\phi_{j,q_j}$, with
    $$\phi_{j,l} = g_{j,l}(y_{(x_0,\theta_0),1}^{(0)}, \dots, y^{(d_m')}_{(x_0,\theta_0),m}), \quad l \in \{1,\dots,q_j\},$$
    are linearly independent in $[a,b)$ and we can conclude analogously to the proof of Theorem \ref{det_univocally}, along with Remark \ref{range_qi_times}, choosing in this case $t_{1,1}, \dots, t_{p,q_p} \in [a,b)$.

    On the other hand, to recover the initial condition, we proceed as in the proof of Theorem \ref{det_univocally}.

    Hence, we are able to recover $(x_0,\theta_0)$ univocally when knowing $y_{(x_0,\theta_0)}(t)$, $t \in [a,b)$, using its values and the values of its derivatives at some finite set of times in $[a,b)$; concretely, between $\tilde{q}$ and $q$ different suitable times in $[a,b)$.
\end{proof}

\begin{remark}\label{ydiscret2}
    Recall that, as in Remark \ref{range_qi_times}, we may not need to know $y_{(x_0,\theta_0)}$ in some interval $[a,b)$, but only the values of this function and its derivatives at the times indicated in the proof of Lemma \ref{recover_infi_lindep}, where the needed order of the derivatives is given by the same lemma.
\end{remark}

\begin{remark}\label{rem_wron}
    The needed amount of times required in Lemma \ref{recover_infi_lindep} could be reduced to 1 if we used higher derivatives of $y$. This procedure is very similar to what is classically performed, and hence we do not treat this case in this document.
\end{remark}

Up to now, we have provided some hypotheses that assure a system is identifiable and observable. Besides, if we have some analyticity properties, we can have data only in an interval $[a,b) \subset \mathcal{I}$ as small as desired. Let us illustrate it in the following examples.

\section{Illustrative examples}\label{illexmp}

In all of the following examples we will need to consider that we have non-constant observations. Hence, for simplicity of the arguments, we will systematically omit solutions from the state spaces such that the observation variables are constant. Such solutions will typically be steady states, but not only.

\subsection{Linearly parameterized rational systems}

We revisit the class of nonlinear systems considered in \cite{Gevers2016}:
\begin{equation}\label{linearlyparam}
    \begin{array}{lcl}
        \dot{x} & = & \dfrac{1}{n(x)} \left(\theta^\mathrm{T}\varphi(x) + \displaystyle\sum_{i=0}^m \rho_i(x)u^i\right),\\
        y &=& x,
    \end{array}
\end{equation}
where $x(0) = x_0 \in \Omega \subset \RR$, $\Omega$ positively invariant with respect to the ODE of System \eqref{linearlyparam}; $\theta \in \Theta = \RR^b$ is the vector of unknown parameters; every component $\varphi_j$ of $\varphi = (\varphi_1, \dots, \varphi_b)^{\mathrm{T}}$ is a polynomial function; $n$, $\rho_i$, $i \in \{0,\dots,m\}$, are also polynomial functions such that $n(x) > 0$, $x \in \Omega$; and $u$ is a scalar time function (which acts as a control in \cite{Gevers2016}). It is clear that, if $t \mapsto u(t)$ is analytic (as required in \cite{Gevers2016}), for any $(\xi,\theta) \in \Omega\times\Theta$, then, as the right-hand side of the ODE is locally Lipschitz-continuous w.r.t. $x \in \Omega$, the corresponding solution $x(\cdot)$ is unique and analytic. Note that one can recover the general autonomous framework considering that $t$ is another (known) state variable such that $\dot{t} = 1$.

The authors study the identifiability of System \eqref{linearlyparam} with respect to $\theta$.

Let $\varphi = (\varphi_1, \dots, \varphi_b)^{\mathrm{T}}$ and $s$ be the highest degree among the $\varphi_j$, $j \in \{1,\dots,b\}$. One has
$$\varphi_j(x) = \sum_{i=0}^s a_{i,j}x^i, \quad j \in \{1,\dots,b\},$$
for some $a_{i,j} \in \RR$, $i \in \{0,\dots,s\}$, $j \in \{1,\dots,b\}$. Let $A = (a_{i,j})_{i = 0,\dots,s,\ j = 1,\dots,b}$.

We will need the following Assumptions \ref{varphidegree} and \ref{x_nonconstant}, which also play a role in \cite{Gevers2016}:
\begin{assum}\label{varphidegree}
   The highest degree $s$ of polynomials $\varphi_j$, $j \in \{1,\dots,b\}$, satisfies $s \geq b-1$.
\end{assum}

\begin{assum}\label{x_nonconstant}
    We assume that the control $u(\cdot)$ is an analytic function in time such that the solution of the ODE in System \eqref{linearlyparam} is defined for all $t \geq 0$ and not constant.
\end{assum}

Then, we claim the following:
\begin{prop}
    Under Assumptions \ref{varphidegree} and \ref{x_nonconstant}, System \eqref{linearlyparam} is identifiable on $\Theta$ with initial conditions in $\Omega$ if, and only if, $A$ has full rank. Moreover, we can determine $\theta$ univocally in terms of $y$, $u$ and $\dot{y}$.
\end{prop}

\begin{proof}
Let us rewrite the ODE equation in System \eqref{linearlyparam} in the following way:
\begin{equation}\label{linear_x}
    n(x)\dot{x} - \displaystyle\sum_{i=0}^s \rho_i(x)u^i = \displaystyle\sum_{i=0}^s \left(\sum_{j=1}^b \theta_ja_{i,j}\right)x^i = (1,x,\dots,x^s)A\theta,
\end{equation}
i.e., opposed to the linearly parameterized expression considered in System \eqref{linearlyparam}, we sort the expression in terms of monomials of $x$, which will be more suitable for our methodology. Then, notice first that, if $A$ does not have full rank, for any $\theta \in \Theta$, there exists some $\tilde{\theta} \in \Theta$ such that $A\theta = A\tilde{\theta}$. Hence, for System \eqref{linearlyparam} to be identifiable, we need that $A$ has full rank.

Let us now differentiate $y$ once and substitute $x$ and $\dot{x}$ in \eqref{linear_x} to obtain
\begin{equation}\label{linear_y}
    n(y)\dot{y} - \sum_{i=0}^s \rho_i(y)u^i = \sum_{i=0}^s \left(\sum_{j=1}^b \theta_ja_{i,j}\right)y^i.
\end{equation}
Then, we have an equation in the form of \eqref{linear_eq}, i.e., point 1 of Theorem \ref{generalization_ident} is satisfied. Let
$$r(\theta) = \left(\sum_{j=1}^b \theta_ja_{0,j}, \dots, \sum_{j=1}^b \theta_ja_{s,j}\right) = A\theta.$$
This implies that point 2 of Theorem \ref{generalization_ident}, i.e., the injectivity of $r$, is also satisfied if $A \in \RR^{(s+1)\times b}$ has full rank. Finally, proving point 3 of Theorem \ref{generalization_ident} is straightforward, since equation \eqref{linear_y} is a polynomial on $y$, and $y$ is analytic non-constant due to Assumption \ref{x_nonconstant}, so $1, y, \dots, y^s$ are linearly independent. Hence, we can recover univocally $\theta$ differentiating the observations just once, instead of differentiating $s$ times as performed in \cite{Gevers2016}.

Then, we need to find at most $s+1$ different times $t_1, \dots, t_{s+1}\geq 0$, which can be chosen in some $[a,b) \subset [0,\infty)$ as small as desired, such that we can solve univocally the following linear system:
$$\begin{pmatrix}
    1 & y(t_1) & \dots & y^s(t_1) \\
    \vdots & \vdots & \ddots & \vdots \\
    1 & y(t_{s+1}) & \dots & y^s(t_{s+1})
\end{pmatrix}\begin{pmatrix}
    \sigma_1 \\
    \vdots \\
    \sigma_{s+1}
\end{pmatrix} = \begin{pmatrix}
    \psi(t_1) \\
    \vdots \\
    \psi(t_{s+1})
\end{pmatrix},$$
where $\psi(t_l) = n(y(t_l))\dot{y}(t_l) - \sum_{i=0}^s \rho_i(y(t_l))u(t_l)^i$, $l \in \{1,\dots,s+1\}$.
This implies that, if $A$ has full rank, there exists some submatrix $\tilde{A} \in \RR^{b\times b}$ of $A$ such that, given $r(\theta) = \sigma$, it implies $\theta = \tilde{A}^{-1}\sigma$.
\end{proof}

\subsection{A non-rational system}

We revisit now Example 4 of \cite{Jain2019}, which is the motivating example the authors use to discuss the approach proposed in \cite{Saccomani2003} based on differential algebra.

Let us consider the following model for a non-isothermal reactor system:
\begin{equation}\label{nonisothermal}
    \begin{array}{lcl}
        \dot{c}_A & = & -k_{10}\mathrm{e}^{-\frac{E}{T}}c_A, \\
        \dot{c}_B & = & k_{10}\mathrm{e}^{-\frac{E}{T}}c_A, \\
        \dot{T} & = & -h_1k_{10}\mathrm{e}^{-\frac{E}{T}}c_A, \\
        y_1 & = & c_A, \\
        y_2 & = & T,
    \end{array}
\end{equation}
for $(c_A(0), c_B(0), T(0))^{\mathrm{T}} \in \Omega = (0,\infty)\times[0,\infty)\times(0,\infty)$ and $(k_{10}, h_1, E)^{\mathrm{T}} \in \Theta = (0,\infty)^3$. The set $\Omega$ is clearly positively invariant for System \eqref{nonisothermal}. We aim to recover the parameters $\theta = (k_{10}, h_1, E)^{\mathrm{T}}$, which the authors do through several computations using Padé approximation and differential algebra.

For our approach, we differentiate $y_1$ and $y_2$ once, obtaining
$$\dot{y}_1 = -k_{10}\mathrm{e}^{-\frac{E}{y_2}}y_1, \quad \dot{y}_2 = -h_1k_{10}\mathrm{e}^{-\frac{E}{y_2}}y_1.$$
Then, after some computations, we obtain the following relations in the form of \eqref{linear_eq} in point 1 of Theorem \ref{generalization_ident}:
\begin{equation}\label{lineqiso}
    \log(-\dot{y}_1) - \log(y_1) = \log(k_{10}) - \dfrac{E}{y_2}, \quad \dfrac{\dot{y}_1}{\dot{y}_2} = h_1.
\end{equation}
Let now
$$r(\theta) = \left(\log(k_{10}), -E, h_1\right),$$
which is clearly injective from $\Theta$ to $r(\Theta)$, satisfying point 2 of Theorem \ref{generalization_ident}. Finally, checking point 3 of Theorem \ref{generalization_ident} consists on proving the linear independence when we consider $(\xi,\theta) \in \Omega\times\Theta$ only of $1$ and $1/y_2$, since it is trivial for the second equation in \eqref{lineqiso}. The linear independence of $1$ and $1/y_2$ is also straightforward since $y_2$ is a non-constant function.

Then, the hypotheses of Theorem \ref{generalization_ident} are satisfied differentiating only once, and we can therefore identify $k_{10}$, $h_1$ and $E$. For $h_1$, it is clear that, for any $t \geq 0$, $h_1 = \dot{y}_1(t)/\dot{y}_2(t)$. On the other hand, we need to find $t_1, t_2 \geq 0$ such that we can solve univocally:
$$\begin{pmatrix}
    1 & \dfrac{1}{y_2(t_1)} \\[1em]
    1 & \dfrac{1}{y_2(t_2)}
\end{pmatrix}\begin{pmatrix}
    \sigma_1 \\
    \sigma_2
\end{pmatrix}=\begin{pmatrix}
    \log(-\dot{y}_1(t_1))-\log(y_1(t_1)) \\
    \log(-\dot{y}_1(t_2))-\log(y_1(t_2))
\end{pmatrix},$$
and then
$$k_{10} = \mathrm{e}^{\sigma_1}, \quad E = -\sigma_2.$$
For this particular case, it can be easily seen that $y_2$ is strictly monotonic in $\Omega$, and hence we can choose $t_1$ and $t_2$ in $[0,\varepsilon)$, for any $\varepsilon > 0$.

Let us underline that we do not make any approximation here, opposed to what the authors perform in \cite{Jain2019}.

\subsection{Hénon-Heiles system}

In \cite{Holmsen2024}, the authors propose a numerical algorithm for identification of pseudo-Hamiltonian systems without an analytical study of identifiability. We revisit the example they propose in 5.1, which is a pure, separable Hamiltonian system. Let us consider we do not know the coefficients of the following Hamiltonian associated to the Hénon-Heiles system, a well-known model in astronomy (see \cite{Henon1964}):
\begin{equation}\label{hamiltonian}
    H(q,p) = a_1q_1^2 + a_2q_2^2 + a_3p_1^2 +  a_4p_2^2 + a_5q_1^2q_2 + a_6q_2^3,
\end{equation}
with $(q,p) \in \RR^2\times\RR^2$, $\theta = (a_1,\dots,a_6) \in \Theta = \{\theta \in \RR^6 : \theta_i \neq 0, \ i = 1,\dots,6\}$. We consider the following system (the Hénon-Heiles system) derived from $H$, along with the observation of both $q$ and $p$:
\begin{equation}\label{henonheiles}
    \begin{array}{lclcl}
        \dot{p}_1 & = & \phantom{-}\partial H/\partial q_1 & = & 2a_1q_1 + 2a_5q_1q_2, \\
        \dot{p}_2 & = & \phantom{-}\partial H/\partial q_2 & = & 2a_2q_2 + a_5q_1^2 + 3a_6q_2^2, \\
        \dot{q}_1 & = & -\partial H/\partial p_1 & = & -2a_3p_1, \\
        \dot{q}_2 & = & -\partial H/\partial p_2 & = & -2a_4p_2, \\
        y_1 & = & q_1, & & \\
        y_2 & = & q_2, & & \\
        y_3 & = & p_1, & & \\
        y_4 & = & p_2, & &
    \end{array}
\end{equation}
considering $(q(0),p(0)) \in \Omega = \RR^4\setminus\{E_1, E_2, E_3, E_4\}$, where $E_1, E_2, E_3, E_4$ are the four different equilibria that the dynamical system of System \eqref{henonheiles} can have. This way, $y_1, y_2, y_3, y_4$ are non-constant. We aim to recover the parameters $\theta = (a_1, a_2, a_3, a_4, a_5, a_6)^{\mathrm{T}}$, all of them being assumed non-null. We start differentiating $y_1, y_2, y_3, y_4$, obtaining
\begin{equation}\label{y_hamil}
    \dot{y}_1 = -2a_3y_3, \quad \dot{y}_2 = -2a_4y_4, \quad \dot{y}_3 = 2a_1y_1 + 2a_5y_1y_2, \quad \dot{y}_4 = 2a_2y_2 + a_5y_1^2 + 3a_6y_2^2.
\end{equation}
We obtain four relations in the form of \eqref{linear_eq} in point 1 of Theorem \ref{generalization_ident}. Let now
$$r(\theta) = (-2a_3, -2a_4, 2a_1, 2a_5, 2a_2, a_5, 3a_6),$$
which is injective in $\Theta$, satisfying point 2 of Theorem \ref{generalization_ident}. Finally, we need to check point 3 of Theorem \ref{generalization_ident}, i.e., whether different functions are linearly independent considering $(\xi,\theta) \in \Omega\times\Theta$. This point is straightforward for the two first equations in \eqref{y_hamil} since $y_3 \not \equiv 0$ and $y_4 \not \equiv 0$. We need now to check the linear independence of the functions in the sets $G_1 = \{y_1, y_1y_2\}$ and $G_2 = \{y_2, y_1^2, y_2^2\}$ whenever $(\xi,\theta) \in \Omega\times\Theta$:
\begin{itemize}
    \item For set $G_1$: let $a, b \in \RR$ such that $ay_1 + by_1y_2 = 0$, which, in terms of $q, p$ is $aq_1 + bq_1q_2 = 0$. Since $q_1 \not \equiv 0$ in $\Omega$, it is equivalent to $a + bq_2 = 0$. Then, since we consider $q_2$ non-constant, the functions in set $G_1$ are linearly independent when $(\xi,\theta) \in \Omega\times\Theta$.
    \item For set $G_2$: we have not been able to prove the linear independence of $y_2, y_1^2, y_2^2$. Nevertheless, if we consider only the other three equations in \eqref{y_hamil} we have just proved that, in particular, we can identify $2a_1$ and $2a_5$ using different times $t_{3,1}, t_{3.2} \geq 0$ such that we can solve the following system:
    $$\begin{pmatrix}
        y_1(t_{3,1}) & y_1(t_{3,1})y_2(t_{3,1}) \\
        y_1(t_{3,2}) & y_1(t_{3,2})y_2(t_{3,2})
    \end{pmatrix}\begin{pmatrix}
        \sigma_1 \\
        \sigma_2
    \end{pmatrix} = \begin{pmatrix}
        \dot{y}_3(t_{3,1}) \\
        \dot{y}_3(t_{3,2})
    \end{pmatrix}.$$
    Then, $a_1 = \sigma_1/2$ and $a_5 = \sigma_2/2$. In particular, we can rewrite the fourth equation in \eqref{y_hamil} as
    $$\dot{y}_4 - \dfrac{\sigma_2}{2}y_1^2 = 2a_2y_2 + 3a_6y_2^2,$$
    which is also in the form of \eqref{linear_eq}, and hence we just need to prove the linear independence of $y_2$ and $y_2^2$, which is equivalent to prove the linear independence of $q_2$ and $q_2^2$. This is straightforward since we are not considering equilibrium points and $q_2$ is analytic.
\end{itemize}

Then, the hypotheses of Theorem \ref{generalization_ident} are satisfied if we consider the equations
\begin{equation}\label{y_hamil2}
    \dot{y}_1 = -2a_3y_3, \quad \dot{y}_2 = -2a_4y_4, \quad \dot{y}_3 = 2a_1y_1 + 2a_5y_1y_2, \quad \dot{y}_4 - \dfrac{\sigma_{3,2}}{2}y_1^2 = 2a_2y_2 + 3a_6y_2^2,
\end{equation}
taking into account what explained previously concerning $\sigma_{3,2}$, and
$$\tilde{r}(\theta) = (-2a_2, -2a_4, 2a_1, 2a_5, 2a_2, 3a_6).$$
Then, on one hand, for almost any $t \geq 0$, $a_3 = -\dot{y}_1(t)/(2y_3(t))$ and $a_4 = -\dot{y}_2(t)/(2y_4(t))$. On the other hand, there exist different $t_{3,1}, t_{3,2} \geq 0$ and different $t_{4,1}, t_{4,2} \geq 0$, which can be chosen in any $[a,b) \subset \mathcal{I}$, such that we can solve univocally the following systems:
$$\begin{array}{rcl}\begin{pmatrix}
    y_1(t_{3,1}) & y_1(t_{3,1})y_2(t_{3,1}) \\
    y_1(t_{3,2}) & y_1(t_{3,2})y_2(t_{3,2})
\end{pmatrix}\begin{pmatrix}
    \sigma_{3,1} \\
    \sigma_{3,2}
\end{pmatrix}&=&\begin{pmatrix}
    \dot{y}_3(t_{3,1}) \\
    \dot{y}_3(t_{3,2})
\end{pmatrix}, \\[1.5em]
\begin{pmatrix}
    y_2(t_{4,1}) & y_2^2(t_{4,1}) \\
    y_2(t_{4,2}) & y_2^2(t_{4,2})
\end{pmatrix}\begin{pmatrix}
    \sigma_{4,1} \\
    \sigma_{4,2}
\end{pmatrix}&=&\begin{pmatrix}
    \dot{y}_4(t_{4,1}) - \dfrac{\sigma_{3,2}}{2}y_1^2(t_{4,1}) \\[0.75em]
    \dot{y}_4(t_{4,2}) - \dfrac{\sigma_{3,2}}{2}y_1^2(t_{4,2})
\end{pmatrix}.
\end{array}$$
Then,
$$a_1 = \frac{\sigma_{3,1}}{2}, \quad a_5 = \frac{\sigma_{3,2}}{2}, \quad a_2 = \frac{\sigma_{4,1}}{2}, \quad a_6 = \frac{\sigma_{4,2}}{3},$$
i.e., the parameters of model \eqref{henonheiles} are identifiable and can be recovered through our methodology.

\subsection{A Lotka-Volterra system}

We consider a classical Lotka-Volterra model along with the observation of the predation term (i.e., the death of the preys by the predators) and the natural death of the predators, and we will aim to reconstruct both states and all the parameters:
\begin{equation}\label{LVsys}
    \begin{array}{lcl}
        \dot{x}_1 & = & \alpha x_1 - \beta x_1x_2, \\
        \dot{x}_2 & = & \gamma x_1x_2 - \delta x_2, \\
        y_1 & = & \beta x_1x_2, \\
        y_2 & = & \delta x_2,
    \end{array}
\end{equation}
with $(x_1(0),x_2(0))^{\mathrm{T}} \in \Omega = (0,\infty)^2\setminus\{(\delta/\gamma,\ \beta/\alpha)^{\mathrm{T}}\}$ and $\theta = (\alpha, \beta, \gamma, \delta)^{\mathrm{T}} \in \Theta = (0,\infty)^4$. The set $\Omega$ is positively invariant with respect to the system of ODEs of System \eqref{LVsys}. First, for any $\theta \in \Theta$, the map
$$\xi \mapsto (y_{(\xi,\theta),1}(0), y_{(\xi,\theta),2}(0)) = (\beta\xi_1\xi_2, \delta \xi_2)$$
is injective in $\Omega$. Then, due to Theorem \ref{generalization_obs}, System \eqref{LVsys} is observable on $\Omega$ in any $[a,b) \subset \mathcal{I}$ with parameters in $\Theta$. Next, if we differentiate once $y_1$ and $y_2$, we obtain, after some computations
\begin{equation}\label{first_dy}
    \dot{y}_1 = (\alpha-\delta)y_1 - \dfrac{\beta}{\delta}y_1y_2 + \dfrac{\gamma\delta}{\beta}\dfrac{y_1^2}{y_2}, \quad \dot{y}_2 = \dfrac{\gamma\delta}{\beta}y_1 - \delta y_2.
\end{equation}
We have two equations in the form of \eqref{linear_eq}, i.e., point 1 of Theorem \ref{generalization_ident} is satisfied. Let
$$r(\theta) = \left(\alpha - \delta, -\dfrac{\beta}{\delta}, \dfrac{\gamma\delta}{\beta}, \dfrac{\gamma\delta}{\beta}, -\delta\right).$$
It is easy to prove point 2 of Theorem \ref{generalization_ident}, i.e., that $r$ is injective in $\Theta$. Finally, we can check point 3 of Theorem \ref{generalization_ident}, which consists on proving that the functions of the sets $G_1 = \{y_1, y_1y_2, y_1^2/y_2\}$ and $G_2 = \{y_1, y_2\}$ are linearly independent when we consider $(\xi,\theta)^{\mathrm{T}} \in \Omega\times\Theta$.
\begin{itemize}
    \item For set $G_1$: let $a$, $b$, $c \in \mathbb{R}$ such that $ay_1 + by_1y_2 + cy_1^2/y_2 = 0.$ We can write this equation in terms of $x_1$ and $x_2$, and, given that $x_1, x_2 \not \equiv 0$, obtain
    \begin{equation*}
        a\beta + b\beta\delta x_2 + c\dfrac{\beta^2}{\delta}x_1 = 0,
    \end{equation*}
    i.e., we need to prove whether $\{1, x_1, x_2\}$ are linearly independent in $\Omega$. Proving that $a = b = c = 0$ for solutions in $\Omega$ is equivalent to prove $A = B = C = 0$ for solutions in $\Omega$ in the equation $A + Bx_1 + Cx_2 = 0$. If $C = 0$, then $x_1$ must be constant, but we have eliminated all solutions in $\Omega$ such that $x_1$ is constant. Hence, let $C \neq 0$. Then, we can rewrite $x_2 = \tilde{A} + \tilde{B}x_1$, and we have
    \begin{align*}
    \dot{x}_2 = \tilde{B}\dot{x}_1 = \tilde{B}x_1 \left(\alpha -\beta(\tilde{A} + \tilde{B}x_1)\right) \quad \text{and} \quad \dot{x}_2 = (\tilde{A} + \tilde{B}x_1) (\gamma x_1 -  \delta).
    \end{align*}
    If we equal both expressions, we obtain that
    $$(\tilde{B}\gamma - \tilde{B}^2\beta)x_1^2 + (\tilde{A}\gamma - \tilde{B}\delta - \tilde{B}\alpha + \tilde{A}\tilde{B}\beta)x_1 - \tilde{A}\delta = 0.$$
    This can only happen if $x_1$ is constant, which can not occur in $\Omega$, or the coefficients are null, which is easy to see that it implies that $\tilde{A} = \tilde{B} = 0$, and hence $A = B = C = 0$.
    \item For set $G_2$: let $a, b \in \RR$ such that $ay_1 + by_2 = 0$. This equation in terms of $x_1$ and $x_2$ is
    $$a\beta x_1x_2 + b\delta x_2 = 0,$$
    i.e., we need to prove whether $\{x_1x_2, x_2\}$ are linearly independent in $\Omega$, which is straightforward, since $x_2 \not \equiv 0$ implies $a\beta x_1 + b\delta = 0$. This implies that $x_1$ is constant, which cannot happen in $\Omega$, or $a = b = 0$.
\end{itemize}

Then, the hypotheses of Theorem \ref{generalization_ident} are satisfied and, hence, we can identify $\alpha$, $\beta$, $\gamma$ and $\delta$. In particular, we just have to select different $t_{1,1}, t_{1,2}, t_{1,3} \geq 0$ and $t_{2,1}, t_{2,2} \geq 0$ such that we can solve univocally the following linear systems:
$$\begin{array}{rcl}
    \begin{pmatrix}
        y_1(t_{1,1}) & \psi_1(t_{1,1}) & \psi_2(t_{1,1}) \\[1em]
        y_1(t_{1,2}) & \psi_1(t_{1,2}) & \psi_2(t_{1,2}) \\[1em]
        y_1(t_{1,3}) & \psi_1(t_{1,3}) & \psi_2(t_{1,3})
    \end{pmatrix}\begin{pmatrix}
        \sigma_{1,1} \\
        \sigma_{1,2} \\
        \sigma_{1,3}
    \end{pmatrix} &=& \begin{pmatrix}
        \dot{y}_1(t_{1,1}) \\
        \dot{y}_1(t_{1,2}) \\
        \dot{y}_1(t_{1,3})
    \end{pmatrix}, \\[4em]
    \begin{pmatrix}
        y_1(t_{2,1}) & y_2(t_{2,1}) \\
        y_1(t_{2,2}) & y_2(t_{2,2})
    \end{pmatrix}\begin{pmatrix}
        \sigma_{2,1} \\
        \sigma_{2,2}
    \end{pmatrix} &=& \begin{pmatrix}
        \dot{y}_2(t_{2,1}) \\
        \dot{y}_2(t_{2,2})
    \end{pmatrix},
\end{array}$$
where $\psi_1(t_{1,l}) = y_1(t_{1,l})y_2(t_{1,l})$ and $\psi_2(t_{1,l}) = \dfrac{y_1^2(t_{1,l})}{y_2(t_{1,l})}$, $l \in \{1,2,3\}$. Then, $r(\alpha, \beta, \gamma, \delta)=(\sigma_1, \sigma_2)$ implies
$$\delta = -\sigma_{2,2}, \quad \alpha = \sigma_{1,1} + \delta, \quad \beta = -\delta\sigma_{1,2}, \quad \gamma = \dfrac{\beta}{\gamma}\sigma_{1,3} = \dfrac{\beta}{\gamma}\sigma_{2,1},$$
i.e., the parameters are determined using only one derivative of each observation.

Analogously to previous examples, this could have been done considering observations in any $[a,b) \subset [0,\infty)$.

\section{Conclusions}

In this work, we have proposed a new methodology to study the identifiability and observability of a class of autonomous dynamical systems. Moreover, we provide different constructive procedures to recover the unknowns, depending on the analyticity of the system and the observations. In particular, we have shown on several cases that this method may require less derivation of the output compared to other strategies which require higher differentiation of the observation variables.

\section*{Acknowledgments}                             
Partially supported by the Spanish Government under projects PID2019-106337GB-I00, PID2023-146754NB-I00, the European M-ERA.Net under project PCI2024-153478, and the French National Research Agency (ANR) under project ANR NOCIME (New Observation and Control Issues Motivated by Epidemiology) 2024-27. The author A.B. Kubik has been also supported by an FPU predoctoral grant and a mobility grant, both of the Ministry of Universities of the Spanish Government.

%\newpage

\bibliographystyle{ieeetr}
\bibliography{mybibfile}

%\printbibliography          

\end{document}